\documentclass{amsart}

\usepackage{amssymb}
\usepackage[T1]{fontenc}
\usepackage{lmodern}
\usepackage{comment}

\usepackage[francais,english]{babel}

\usepackage{latexsym}

\usepackage[centering,textheight=50.5pc,textwidth=30pc]{geometry}
\usepackage[pdfstartview=XYZ]{hyperref}
\theoremstyle{plain}
\newtheorem{theorem}{Theorem}[section]
\newtheorem{proposition}[theorem]{Proposition}
\newtheorem{lemma}[theorem]{Lemma}

\theoremstyle{definition}

\theoremstyle{remark}

\theoremstyle{conjecture}
\newtheorem{conjecture}[theorem]{Conjecture}

\theoremstyle{condition}

\title[Special perfects]{On odd perfect numbers of special forms}

\date{\today}

\author[L. H. Gallardo]{Luis H. Gallardo}

\address{Univ. Brest,
UMR CNRS 6205\\
Laboratoire de Math\'ematiques de Bretagne Atlantique
6, Av. Le Gorgeu, C.S. 93837,  Cedex 3, F-29238 Brest\\
France.}

\email{Luis.Gallardo@univ-brest.fr}

\author[O. Rahavandrainy]{Olivier Rahavandrainy}

\address{Univ. Brest,
UMR CNRS 6205\\
Laboratoire de Math\'ematiques de Bretagne Atlantique
6, Av. Le Gorgeu, C.S. 93837,  Cedex 3, F-29238 Brest\\
France.}

\email{Olivier.Rahavandrainy@univ-brest.fr}

\subjclass[2010]{Primary 11A25; Secondary 11A07}

\keywords{Sums of divisors, odd perfect numbers, cyclotomic polynomials, congruences.}

\begin{document}


\begin{abstract}
We give necessary conditions for perfection of some families of odd numbers with special multiplicative forms. Extending earlier work of Steuerwald, Kanold, McDaniel et al.
\end{abstract}

\maketitle

\section{Introduction}
\label{intro}
A number $n$ is perfect if and only if it is equal to the sum of all its proper divisors. Let $n$ be an odd perfect number. Euler \cite{sept} proved
\begin{equation}
n = p^{4k+1} c^2
\end{equation}
where $p$ is a prime number congruent  to $1$ modulo $4$, $k \geq 0$ is a nonnegative integer and $c >0$ is a positive integer non divisible by $p$. Put $\alpha = 4k+1$, and $c = a_1^{e_1} \cdots  q_t^{e_t}$ with prime numbers $q_1,\ldots,q_t$ and positive integers $e_1,\ldots,e_t$. If all $e$'s are equal we put $\beta = e_1 = \ldots = e_t$, if all but one are equal, we put, say, $\gamma = e_2= \ldots = e_t$. If all  $e$'s are congruent to $a$ modulo $m$ we write $\delta \equiv e_1 \ldots \equiv e_t \equiv a \pmod m$. Put $g = \gcd(2e_1+1,\ldots,2e_t+1)$.

Steuerwald \cite{dixneuf} proved that $\beta \neq 1$, Kanold \cite{once} proved that $\beta \neq 2$, and that $g^4$ divides $n$. He deduced that none of $9,15,21,33$ can divide $g$. Brauer \cite{uno} and Kanold \cite{douze} independently proved that $e_1=2,\gamma=1$, is also impossible; see Brauer's discussion of the two proofs in \cite{dos}. Kanold \cite{treize} announced and Kanold \cite{catorce} proved that when $\gamma = 1$ one has $e_1 \notin \{3,4\}$. Kanold \cite{quinze} proved that both cases $e_1=e_2=2,e_3= \ldots = e_t = 1$ and $\alpha = 5, e_1=\ldots = e_r = 1,e_{r+1}= \ldots = e_t =2$ are impossible. McCarthy \cite{dieciseis} proved that $3 \nmid n$ and $\delta \equiv 1 \pmod{3}$ cannot hold. He proved also in the same paper that $3 \nmid n, \gamma = 1$ and $q_1 \equiv 2 \pmod{3}$ is impossible. 

McDaniel proved \cite{dixsept} that  $\delta \equiv 1 \pmod{3}$ cannot hold. Hagis and McDaniel \cite{nueve} proved that $\beta \neq 3$ and that ($q_1=7, e_1=6$ and $7$ divides $e_2,\ldots, e_t$) is impossible. Hagis and McDaniel \cite{diez} proved that $\beta \notin \{5,12,17,24,62\}$. Hagis and McDaniel \cite{diez} proposed the still unresolved conjecture that there are no odd perfect numbers with all $e$'s equal. Cohen and Williams \cite{tres} proved that $\beta \notin \{6,8,11,14,18\}$. Thus, all cases of $\beta$ less than $20$ are resolved. They also proved that for $\gamma = 1$ one has $e_1 \notin \{5,6\}$, extending earlier results.

Recently, Evans and Pearlman \cite{six} proved the two facts: i) $\delta \equiv 32 \pmod{65}$ and ii) $7 \mid n, \delta \equiv 2 \pmod{5}$, are impossible. Recently also, Gallardo \cite{ocho} proved that $n$ cannot be a cube.

See also \cite[Chapter 1]{cuatro} for a comprehensive survey on this area.

As usual $a \mid \mid b$ means that $a$ divides $b$ (denoted also by $a \mid b$) and that $\gcd(a,b/a) = 1$. For a positive integer $n >0$, we denote by $\Phi_n(x)$ the $n$-th cyclotomic polynomial. We denote by $\sigma(n)$ the sum of all positive divisors of $n$. We denote by $\omega(n)$ the number of distinct prime factors of $n$. For a finite set $A$, we denote by $\sharp(A)$ the number of elements of $A$.

The object of this paper is to prove the following two results:
\begin{theorem}
\label{th1d2}
There is no odd perfect number of the form
\begin{equation*}
n = 5^{\alpha} M^{2 \beta}
\end{equation*}
where $M$ is square-free and $5 \nmid M$.
\end{theorem}

In order to present our second result, we need the following definitions. Let $\Sigma = \{p\;\text{prime}\; \colon  p \in \{2,4\} \pmod{7}$. For any prime number $x > 7$ let $T(x) = \{q\;\text{prime}\;\colon q \neq 3\;\text{and}\;q \mid \Phi_3(x)\}$. Let $p > 7$ be a prime number. We say that $p$ is \emph{good} if for some nonnegative integer $n \geq 0$ the set $S_n(p) \cap \Sigma$ is nonempty, where $S_n(p)$ is a set recursively defined by:
\begin{itemize}
\item[\rm{(a)}]
$S_0(p) = \{p\}$.
\item[\rm{(b)}]
$S_{n+1}(p) = S_n(p) \bigcup_{x \in S_n(p)} T(x)$.
\end{itemize}
Observe that $T(x)$ is not empty, since we claim that there exist no prime number $x$ such that
\begin{equation}
\label{1.3}
\Phi_3(x) = 3^m
\end{equation}
where $m$ is a positive integer. To prove the claim, reduce both sides of \eqref{1.3} modulo $3$ to get $x \equiv 1 \pmod{3}$. So $x \in \{1,4,7\} \pmod{9}$. But $x>1$ so that \eqref{1.3} implies $m > 1$. Reduce now both sides of \eqref{1.3} modulo $9$ to get the contradiction $3 \equiv \Phi_3(x) \equiv 0 \pmod{9}$. This proves the claim.

For example, $x = 31$ is good since $S_0(31) = \{31\}$,$S_1(31) = \{31,331\}$ and $331 \in S_1(31) \cap \Sigma$ since $331 \equiv 2 \pmod{7}$. A more involved example is the fact that $x = 13$ is good; indeed McDaniel's proof of Lemma \ref{L2.3} (in our special case $p = 5$) is essentially the proof that $x = 13$ is good:
\begin{center}
$S_0(13) = \{13\}, S_1(13) = \{13,61\},S_2(13)= \{13,61,97\}$,
\end{center}
\begin{center}
$S_3(13) = \{13,61,97,3169\}, S_4(13) = \{13,61,97,3169,3348577\}$,
\end{center}
\begin{center}
$S_5(13) = \{13,61,97,3169,3348577,3737657091169\}$,
\end{center}
\begin{center}
$S_6(13) \supseteq  R_6 = \{13,61,97,3169,3348577,3737657091169,181\}$,
\end{center}
\begin{center}
$S_7(13) \supseteq  R_7 = R_6 \cup \{79,139\}$,
\end{center}
and $79 \equiv 2 \pmod{7}$ so that $S_7(13) \cap \Sigma$ is a nonempty set. This proves the result.

We are now able to establish our second result.

\begin{theorem}
\label{1.4}
\begin{itemize}
\item[\rm{(a)}]
There is no odd perfect number of the form
\begin{equation*}
n = 5^{\alpha} 3^{2 b} q_1^{6k_1+2}\cdots q_t^{6k_t+2}
\end{equation*}
where at least one of the primes $5 < q_1 <\cdots < q_t$ is good.
\item[\rm{(b)}]
There is no perfect number of the form
\begin{equation*}
n = 5^{\alpha} 3^{2 b} q_1^{6k_1+2}\cdots q_t^{6k_t+2}
\end{equation*}
where $5 < q_1$ and at least one of the primes $q_1 < \cdots < q_t$ does not exceed $157$.
\item[\rm{(c)}]
\end{itemize}
Assuming Conjecture \ref{C1.6}, there is no odd perfect number of the form
\begin{equation*}
n = 5^{\alpha} 3^{2 b} q_1^{6k_1+2}\cdots q_t^{6k_t+2}
\end{equation*}
where $q_1,\ldots,q_t$ are prime numbers exceeding $5$.
\end{theorem}

Some pairs $(\alpha,b)$ for which the condition on part (a) of Theorem \ref{1.4} holds are described in Proposition \ref{P4.1}.

\begin{proposition}
\label{P1.5}
All prime numbers $p > 7$ and less than $160$ are good.
\end{proposition}
\begin{proof}
Follows from some computer calculations.
\end{proof}

\begin{conjecture}
\label{C1.6}
All prime numbers $p > 7$ are good.
\end{conjecture}

\section{Some tools}
\label{toolsBM}

Some classical results follow.
\begin{lemma}
\label{L2.1} \cite{BIB21}
There is no odd perfect number with less than $3$ distinct prime divisors.
\end{lemma}
\begin{lemma}
\label{L2.2} \cite{once}
There is no odd perfect number $n = p^{\alpha} c^2$ where $c$ is square-free.
\end{lemma}
\begin{lemma}
\label{L2.3} \cite{dixsept}
There is no odd perfect number $n = p^{\alpha} q_1^{2e_1}\cdots q_t^{2e_t}$  with $e_1 \equiv e_2 \equiv \cdots \equiv e_t \equiv 1 \pmod{3}$.
\end{lemma}
\begin{lemma}
\label{L2.4} \cite{BIB18}
If $n = p^{\alpha} 3^{2 \beta} d^{2 \beta}$ with square-free $d$ and with $\gcd(3,d) =1 = \gcd(p,d)$ is an odd perfect number then $\sigma(p^{\alpha}) \equiv 0 \pmod{3^{2 \beta}}$.
\end{lemma}
\begin{lemma}
\label{L2.5} \cite[footnote on page 590]{BIB20}
There is no odd perfect number $n$ divisible by $105 = 2 \cdot 5 \cdot 7$.
\end{lemma}

For distinct prime numbers $p \neq q$ we denote by $ord_q(p)$ the smallest natural number $d$ for which $p^d \equiv 1 \pmod{q}$. We denote by $a_q(p)$ the integer $e$ such that $q^e \vert\vert p^d-1$, where $d = ord_q(p)$.

The following lemma appears in Pomerance's paper \cite[Lemma 1]{cinco}.
\begin{lemma}
\label{L2.6}
Suppose that $p,q$ are distinct primes with $q \neq 2$ and $b,c$ are natural numbers. Then
\begin{itemize}
\item[\rm{(a)}]
if $p \equiv 1 \pmod{q}$, then $q^b \vert\vert \sigma(p^c)$ if and only if $q^b \vert\vert c+1$,
\item[\rm{(b)}]
if $p \not \equiv 1 \pmod{q}$, then $q^b \vert\vert \sigma(p^c)$ if and only if $b \geq a_q(p),ord_q(p) \mid c+1$, and $q^{b-a_q(p)} \mid c+1$.
\end{itemize}
\end{lemma}

A recent result of Yamada \cite{veintidos} is
\begin{lemma}
\label{L2.7}
If $n$ is an odd perfect number with $\beta = e_1= \cdots = e_t \geq 1$ then
$$
\omega(n) \leq 4 \beta^2+2 \beta + 3.
$$
\end{lemma}

\section{Proof of Theorem \ref{th1d2}}

Since $6 = \Phi_2(5)$ divides $\sigma(5^{\alpha})$ and $M$ is square-free, one has $M = 3Q$ for some positive square-free integer $Q$ coprime with $3$, so that we can write $n = 5^{\alpha} 3^{2 \beta} Q^{2 \beta}$. Since from Lemma \ref{L2.1} one has $\omega(n) > 2$,it follows that $Q \neq 1$. Observe that $3^{2 \beta +1} \in \{3,2\} \pmod{5}$ so that $5 \nmid \sigma(3^{2 \beta})$. In other words, $5 \mid \sigma(5^{\alpha}) \sigma(3^{2 \beta})$. Thus, since $n$ is perfect we have that $5 \mid \sigma(q^{2 \beta})$ for some prime divisor $q$ of $Q$. 
Set now
\begin{equation}
\label{3.1}
S = \{p \equiv 1 \pmod{5} \colon p \mid Q\}\;\text{and}\; \gamma = \sharp(S).
\end{equation}
Since for each prime divisor $p$ of $Q$ with $p \not \equiv 1 \pmod{5}, d= ord_5(p) \in \{2,4\}$ is even, we have $\gamma > 0$ by Lemma \ref{L2.6}. More precisely, Lemma \ref{L2.6} implies that if a prime divisor $p$ of $Q$ satisfies $p \not \equiv  1 \pmod{5}$ then $5 \nmid \sigma(p^{2 \beta})$ since the even number $d$ cannot divide the odd number $c+1$, where $c = 2 \beta$.
In other words, $q \in S$, so that $S$ is not empty. We can then, by Lemma \ref{L2.6} a), since $q \equiv 1 \pmod{5}$ and since $5 \mid \sigma(q^{2 \beta})$, define a positive integer $b > 0$ by $5^b \vert\vert 2 \beta +1$. By Lemma \ref{L2.6}, and by the definition of $S$ we get $5^b \vert\vert \sigma(p^{2 \beta})$ for all $p \in S$. Thus, since $n$ is perfect we get
\begin{equation}
\label{3.2}
\alpha = b \gamma.
\end{equation}
By Lemma \ref{L2.3} we can assume that $3 \nmid 2 \beta +1$. By Lemma \ref{L2.6} it follows that $3 \nmid \sigma(Q^{2 \beta})$. But since trivially $3 \nmid \sigma(3^{2 \beta})$ and $n$ is perfect we must have $3^{2 \beta} \vert\vert \sigma(5^{\alpha})$. We can also obtain this from Lemma \ref{L2.4}. So by Pomerance's Lemma \ref{L2.6} we get
\begin{equation}
\label{3.3}
3^{2 \beta -1} \vert\vert \alpha + 1.
\end{equation}
Since $b \leq b \cdot \ln(5) \leq \ln(2 \beta +1)$, we get from \eqref{3.2} and \eqref{3.3}
\begin{equation}
\label{3.4}
\gamma \geq \frac{3^{2 \beta -1}-1}{\ln(2 \beta +1)}.
\end{equation}
But, trivially, $\gamma \leq \omega(n)$. So we get from Yamada's Lemma \ref{L2.7} and \eqref{3.4}
\begin{equation}
\label{3.5}
3^{2 \beta -1}-1 \leq \ln(2 \beta +1) \cdot (4 \beta^2+2 \beta +3).
\end{equation}
This forces $\beta \in \{1,2\}$ so that $\beta = 2$ since $\beta \not \equiv 1 \pmod{3}$.
This contradicts Lemma \ref{L2.2} thereby proving the theorem.

\section{Proof of Theorem \ref{1.4}}
\label{S4}

Part c) follows from part a). Part b) follows from part a) and from Proposition \ref{P1.5}. Now we prove part a). Observe that $\Phi_3(y)$ is odd for all integers $y$. Assume that for some $j \in \{1,\ldots,t\}$ the prime $q_j$ is good. Let $k \geq 0$ be a nonnegative integer. Observe that any element of $S_k(q_j)$ is a prime divisor of $n$ since $n$ is perfect. Observe also that for any index $i$, $5 \notin S_k(q_i)$ since for any integer $y$ one has $\Phi_3(y) \not \equiv 0 \pmod{5}$.  Thus, by definition that $q_j$ is good there exists an index $i \in \{1,\ldots,t\}$ and there exists a positive integer $k \geq 0$ such that the prime $q_i$ satisfies $q_i \in S_k(q_j)$, so in particular we obtain that $q_i \mid n$, and, moreover, one has $q_i \in \{2,4\} \pmod{7}$. Thus, $\Phi_3(q_i) \equiv 0 \pmod{7}$.

Since $3 \mid 6 k_i+3$ we have $\Phi_3(q_i) \mid \sigma(q_i^{6 k_i+2})$. Thus, $\Phi_3(q_i) \mid n$ since $n$ is perfect and $\Phi_3(q_i)$ is odd. But $7 \mid \Phi_3(q_i)$. So $7 \mid n$. This is impossible by Lemma \ref{L2.5} since $3$ and $5$ are also divisors of $n$. This proves the theorem.

We give below a sufficient condition in order that the condition on Theorem \ref{1.4} part a) holds.
\begin{proposition}
\label{P4.1}
Let $n$ be an odd perfect number such that
\begin{equation*}
15 \mid n,\;\;5^{\alpha} \vert\vert n\;\;\text{and}\;\;3^{2 b} \vert\vert n
\end{equation*}
for some positive integers $\alpha,b >0$. If
\begin{equation}
\label{E4.2}
3 \mid (\alpha+1)(2 b+1),
\end{equation}
then $n$ has a good prime divisor.
\end{proposition}
\begin{proof}
If $3 \mid \alpha+1$ then one has $31 = \Phi_5(3) \mid \sigma(5^{\alpha})$. Thus, $31 \mid n$ and $31$ is good. If $3 \mid 2 b +1$ then  $13 = \Phi_3(3)  \mid \sigma(3^{2 b})$ so that $13 \mid n$ and $13$ is also good.
\end{proof}


\end{document}